\documentclass[journal,twoside,web]{ieeecolor}

\usepackage{generic}
\usepackage{cite}
\usepackage{amsmath,amssymb,amsfonts}
\usepackage{algorithmic}
\usepackage{graphicx}
\usepackage{textcomp}

\usepackage{graphicx}
\usepackage{epsfig} 
\usepackage{cite}
\usepackage{lcsys}

\usepackage{scalerel}
\usepackage{algorithm}

\newtheorem{theorem}{Theorem}[section]
\newtheorem{lemma}[theorem]{Lemma}
\newtheorem{proposition}[theorem]{Proposition}

\newtheorem{definition}[theorem]{Definition}
\newtheorem{remark}[theorem]{Remark}

\DeclareMathOperator*{\argmin}{arg\,min}

\newcommand{\rr}{{\mathbb R}}
\newcommand{\diag}{\mathop{\rm diag}\nolimits}

\newcommand*{\pdot}{\mathbin{\scalerel*{\boldsymbol\odot}{\circ}}}

\begin{document}

\def\BibTeX{{\rm B\kern-.05em{\sc i\kern-.025em b}\kern-.08em
    T\kern-.1667em\lower.7ex\hbox{E}\kern-.125emX}}
\markboth{\journalname, VOL. XX, NO. XX, XXXX 2017}
{Author \MakeLowercase{\textit{et al.}}: Preparation of Papers for IEEE Control Systems Letters (August 2022)}

\title{Fixed-time-stable ODE Representation of Lasso}

\author{
Liang Wu$^{a}$,
Yunhong Che$^{b}$,
Wallace Gian Yion Tan$^{b,*}$,
Efstathios Iliakis$^{b}$,\\
Richard D. Braatz$^{b}$, Fellow, IEEE,
and J\'an Drgo\v na$^{a}$, Member, IEEE
\thanks{This work was supported by the US DOE, Office of Science, ASCR program under the Scientific Discovery through Advanced Computing (SciDAC) Institute “LEADS: LEarning-Accelerated Domain Science”.}
\thanks{$^{a}$Department of Civil and Systems Engineering,
Johns Hopkins University, Baltimore, MD 21218 USA. $^{b}$Department of Chemical Engineering, Massachusetts Institute of Technology, Cambridge, MA 02139 USA. $^{*}$Wallace Gian Yion Tan was supported by the MathWorks Fellowship. $^{*}$Efstathios Iliakis was supported by the Onassis Foundation Scholarships' Program.Corresponding author: Liang Wu (\texttt{wliang14@jh.edu})}.
}

\maketitle
\thispagestyle{empty}

\begin{abstract}
Lasso problems arise in many areas, including signal processing, machine learning, and control, and are closely connected to sparse coding mechanisms observed in neuroscience. A continuous-time ordinary differential equation (ODE) representation of the Lasso problem not only enables its solution on analog computers but also provides a framework for interpreting neurophysiological phenomena. This article proposes a fixed-time-stable ODE representation of the Lasso problem by first transforming it into a smooth nonnegative quadratic program (QP) and then designing a projection-free Newton-based fixed-time-stable ODE system for solving the corresponding Karush–Kuhn–Tucker (KKT) conditions. Moreover, the settling time of the ODE is independent of the problem data and can be arbitrarily prescribed.
Numerical experiments verify that the trajectory reaches the optimal solution within the prescribed time.
\end{abstract}

\begin{IEEEkeywords}
Sparse optimization, Lasso, Ordinary differential equation, Fixed-time-stable.
\end{IEEEkeywords}

\section{Introduction}
\label{sec:introduction}

\IEEEPARstart{S}{parse} optimization plays a central role in modern signal processing, system identification, and machine learning. Among various formulations, the Lasso problem \cite{tibshirani1996regression}, which augments a least-squares objective with an $\ell_1$-regularization term, has become a cornerstone for promoting sparsity in high-dimensional estimation problems. Lasso problems also arise in model predictive control applications \cite{gallieri2012}.

Despite its conceptual simplicity, the Lasso problem is characterized by the non-smooth nature of the $\ell_1$ regularization, and a large body of literature has focused on designing efficient numerical optimization algorithms (running on digital computers), including proximal gradient methods, coordinate descent, and operator-splitting schemes. However, conventional numerical optimization can be energy-intensive and may scale unfavorably in practice as problem size grows. Moreover, it often lacks certified worst-case execution-time guarantees \cite{wu2025Arbitrarily,Wu2026n3QP}. To address these limitations, our previous work \cite{wu2025Arbitrarily} proposes a valuable concept: transforming convex optimization problems into fixed-time-stable ordinary differential equations (ODEs) and solving them on analog computers can achieve arbitrarily small computational time, independent of problem dimensions. Moreover, establishing the connection between optimization algorithms (or problems) and continuous-time dynamic systems has received increasing attention \cite{wibisono2016variational,boyd2024optimization}. 

This article develops a continuous-time ODE representation for the non-smooth Lasso problem, which is closely related to sparse coding mechanisms observed in biological neural systems. Both theoretical and experimental studies suggest that sparse coding plays an important role in neural information processing \cite{lewicki2002efficient}. In many sensory systems, stimuli are represented by the activity of only a small subset of neurons, leading to sparse population codes. Identifying physical analog architectures, such as ODE-based dynamical systems, that capture this sparse coding mechanism can be useful for understanding and emulating biological neural computation. Since the Lasso formulation is a widely used mathematical model for sparse optimization, deriving an ODE representation for solving the Lasso problem may help test theories that explain neurophysiological phenomena.

\subsection{Related work}
In \cite{rozell2008sparse}, the locally competitive algorithm (LCA) is introduced as a dynamical system-based approach for solving the Lasso problem. Its ODE states are interpreted in terms of membrane potentials and short-term firing rates in a neural population, making the model neurally plausible. The LCA dynamics involving the thresholding function can be realized either through Field Programmable Analog Array \cite{shapero2013configurable}, or through the brain-inspired IBM TrueNorth Neurosynaptic system \cite{fair2019sparse}. The convergence of LCA to a Lasso solution was established in \cite{tang2016convergence}; however, no quantitative analysis of the settling time was provided to determine when the solution can be retrieved.

In \cite{garg2022fixed}, a fixed-time-stable proximal dynamical system for mixed variational inequality problems (including Lasso) is proposed. In \cite{zhou2025fixed}, a fixed-time-stable neurodynamic method for fused Lasso is developed. However, these Lasso-to-ODE approaches \cite{rozell2008sparse,tang2016convergence,garg2022fixed,zhou2025fixed} rely on non-smooth thresholding or projection operators and belong to gradient flows, making the fixed-time stability in \cite{garg2022fixed,zhou2025fixed} dependent on problem data (modulus and Lipschitz constants); see \cite[Thm.~1]{garg2022fixed} and \cite[Thm.~2]{zhou2025fixed}. Our previous work \cite{wu2025Arbitrarily} studies fixed-time-stable ODE representations of general convex optimization via a homogeneous monotone complementarity formulation for infeasibility detection. This article demonstrates that the Lasso case admits a simpler, specialized framework to find its fixed-time-stable ODE representation.

\subsection{Contributions}
This article proposes a Newton-based fixed-time-stable ODE representation for the Lasso problem. The fixed-time stability property is independent of the problem data, allowing the settling time to be arbitrarily prescribed and tuned to small values, as discussed in \cite{wu2025Arbitrarily}. 

This property is achieved through \textit{i)} an equivalent transformation of the Lasso problem into a smooth nonnegative quadratic program (QP), and \textit{ii)} a Newton-based fixed-time-stable ODE for solving the resulting QP, which automatically enforces the non-negativity constraints and thereby avoids non-smooth projection operations.

\section{Fixed-time-stable ODE for Unconstrained Smooth Optimization}
Consider an unconstrained smooth optimization problem,
\begin{equation}\label{eqn_unconstrained_problem}
\min_{x\in\mathbb{R}^{n_x}} f(x),
\end{equation}
where the function $f: \mathbb{R}^{n_x}\rightarrow R$ is convex and differentiable. A point $x^*$ is the global optimal point of $f$ if and only if $\nabla f(x^*) = 0$  (see \cite{boyd2004convex}) and, if $f$ is strongly convex, then $x^*$ is unique. A straightforward approach to convert \eqref{eqn_unconstrained_problem} into an ODE is
\begin{equation}\label{eqn_continuous_dynamic}
    \dot{x} = -k \nabla f(x),
\end{equation}
where $k>0$ is the time-scale parameter. However, this method does not guarantee an upper bound on the settling time.

Before presenting our proposed Analog-flow with certified settling time for problem \eqref{eqn_unconstrained_problem}, we first present some preliminaries about the concept of \textit{finite-time-stable} and \textit{fixed-time-stable}.

\begin{definition}[\textbf{Finite-time-stable} \cite{polyakov2011nonlinear}] Autonomous system \eqref{eqn_continuous_dynamic} is said to be {\em globally finite-time-stable} if it is globally asymptotically Lyapunov stable and reaches the equilibrium $x^*$ at some finite time $\forall x_0\in\rr^{n_x}\backslash\{x^*\}$, where the equilibrium-time function $T(z_0)$ is bounded, i.e., for all $x_0\in\rr^{n_x}$, $\exists~T_{\max}(x_0)>0$ such that $T(x_0)\leq T_{\max}(x_0)$.
\end{definition}

\begin{lemma}[\!\!\cite{Garg2021fixed,polyakov2011nonlinear}] 
    If there exists a positive-definite Lyapunov function $V\in\mathcal{C}^1(\mathcal{D},\rr)$ (where $\mathcal{D} \subset \mathbb{R}^{n_x}$ is a neighborhood of the equilibrium $x^*$) satisfying
    \begin{equation}
        \dot{V}(x)\leq -kV(x)^\alpha, \ \forall x\in\mathcal{D} \backslash \{x^*\},
    \end{equation}
    where the parameters $k>0$ and $\alpha\in(0,1)$, then the settling time $T(x_0)$ for the system \eqref{eqn_continuous_dynamic} can be bounded by
    \vspace{-0.1cm}
\begin{equation}
\label{eqn_finite_time_stable_upper_bound}
        T(x_0)\leq T_{\max}(x_0)=\frac{V(x_0)^{1-\alpha}}{k(1-\alpha)}.
    \end{equation}
\end{lemma}

\vspace{0.15cm}

The drawback of the \textit{finite-time-stable} concept is that, by \eqref{eqn_finite_time_stable_upper_bound}, the upper bound for the settling time $T(x_0)$ is dependent on the initial state $x_0$ and increases without bound when the magnitude of the initial state $\|x_0\|$ increases. To make the settling time independent from the initial state $x_0$, the concept of \textit{fixed-time-stable} is introduced.
\begin{definition}[\textbf{Fixed-time-stable} \cite{polyakov2011nonlinear}] The system \eqref{eqn_continuous_dynamic} is said to be {\em fixed-time-stable} if it is globally finite-time-stable and its equilibrium-time function $T(x_0)$ is globally bounded, i.e., for all $x_0\in\rr^n$, $\exists  T_{\max}\in(0,\infty)$ such that $T(x_0)\leq T_{\max}$.
\end{definition}

\begin{lemma}[\!\!\cite{polyakov2011nonlinear,Garg2021fixed}]\label{lemma_fixed_time_stable}
 If there exists a positive-definite Lyapunov function $V\in\mathcal{C}^1(\mathcal{D},\rr)$ (where $\mathcal{D} \subset \mathbb{R}^n$ is a neighborhood of the equilibrium) satisfying
\begin{equation}\label{eqn_Lyapunov_function_fixed_time_stable}
        \dot{V}(x)\leq -k_1V(x)^{\alpha_1}-k_2V(x)^{\alpha_2}, \ \ \forall x\in\mathcal{D} \backslash \{x^*\},
    \end{equation}
    where the parameters $k_1>0$, $k_2>0$, $0<\alpha_1<1$, and $\alpha_2>1$, then the settling time $T(x_0)$ for the system \eqref{eqn_continuous_dynamic} can be globally bounded by

    \vspace{-0.3cm}
    \begin{equation}\label{eqn_T_max_1}
       T(x_0)\leq T_{\max}=\frac{1}{k_1(1-\alpha_1)}+\frac{1}{k_2(\alpha_2-1)}. 
    \end{equation}
\end{lemma}

    \vspace{0.2cm}

Now the upper bound $T_{\max}$ is independent of the initial point $x_0$, which makes it appealing for achieving the paradigm of arbitrarily small prescribed-time computing. Equation \eqref{eqn_T_max_1} for $T_{\max}$ is not tight for our later presented settings; this article proposes a more accurate estimate as shown in the following proposition.
\begin{proposition}\label{lemma_fixed_time_stable_simplified}
If there exists a Lyapunov function $V$ satisfying \eqref{eqn_Lyapunov_function_fixed_time_stable} with 
\begin{equation}\label{eqn_alpha_1_2}
\alpha_1=1-\frac{1}{\mu}, \qquad \alpha_2=1+\frac{1}{\mu},\qquad\mu>1,    
\end{equation}
then the settling time $T(x_0)$ for the system \eqref{eqn_continuous_dynamic} can be globally bounded by
\begin{equation}\label{eqn_prescribed_time}
         T(x_0)\leq T_{\max} = \frac{\mu\pi}{2\sqrt{k_1k_2}}.
    \end{equation}
Furthermore, the upper bound of the settling time in \eqref{eqn_prescribed_time} is less than \eqref{eqn_T_max_1}.    
\end{proposition}
\begin{proof}
    Consider an auxiliary ODE,
    \[
    \dot{V}=-k_1V^{1-\frac{1}{\mu}}-k_2V^{1+\frac{1}{\mu}}, \quad k_1>0, \ \ k_2>0, \ \mu>1,
    \]
    under the initial condition $V(0)>0$.
    Clearly, $V=0$ is the equilibrium point of this ODE. Applying separation of variables results in
    \[
    t=-\int_{W_0}^{W(t)}\frac{dV}{k_1V^{1-\frac{1}{\mu}}+k_2V^{1+\frac{1}{\mu}}}.
    \]
    Now introduce $W\triangleq V^{\frac{1}{\mu}}$, which also implies that $V=W^\mu$. Thus,
    \[
    \begin{aligned}
        t&=-\int_{W(0)}^{W(t)}\frac{d(W^\mu)}{k_1W^{\mu(1-\frac{1}{\mu})}+k_2W^{\mu(1+\frac{1}{\mu})}} \\
        &=-\mu\!\int_{W(0)}^{W(t)}\frac{W^{\mu-1}dW}{k_1W^{\mu-1}+k_2W^{\mu+1}}=-\frac{\mu}{k_2}\int_{W(0)}^{W(t)}\frac{dW}{\frac{k_1}{k_2}+W^2}\\
        &=-\frac{\mu}{k_2}\!\left[\frac{1}{\sqrt{\frac{k_1}{k_2}}}\arctan\!\left(\frac{W}{\sqrt{\frac{k_1}{k_2}}}\right)\right|_{W(0)}^{W(t)}\\
        &=\frac{\mu}{\sqrt{k_1k_2}}\!\left(\arctan\!\left(\sqrt{\frac{k_2}{k_1}}W(0)\right)\!-\arctan\!\left(\sqrt{\frac{k_2}{k_1}}W(t)\right)\!\right)\!.
    \end{aligned}
    \]
    Thus, when $V(t)\rightarrow0$ (namely $W(t)\rightarrow0$), the required time is
    \[
    t=\frac{\mu}{\sqrt{k_1k_2}}\arctan\!\left(\sqrt{\frac{k_2}{k_1}}W(0)\right).
    \]
    Since $\arctan\!\left(\sqrt{\tfrac{k_2}{k_1}},W(0)\right)\!\leq\tfrac{\pi}{2}$ for any initial value $W(0)$, an upper bound of the settling time is
    \[
    T_{\max}=\frac{\mu\pi}{2\sqrt{k_1k_2}},
    \]
    which completes the first part of the proof.

    Then, when $\alpha_1=1-\frac{1}{\mu},~\alpha_2=1+\frac{1}{\mu},$ the upper bound of the settling time in \eqref{eqn_T_max_1} is larger than that in \eqref{eqn_prescribed_time}, as shown below:
    \[
    \begin{aligned}
         &\frac{\mu (k_1+k_2)}{k_1k_2}=\frac{\mu}{\sqrt{k_1k_2}}\frac{k_1+k_2}{\sqrt{k_1k_2}}=\frac{\mu}{\sqrt{k_1k_2}}\!\left(\sqrt{\frac{k_1}{k_2}}+ \sqrt{\frac{k_2}{k_1}}\right) \\
         &\!\geq\frac{2\mu}{\sqrt{k_1k_2}}>\frac{\mu\pi}{2\sqrt{k_1k_2}}   
    \end{aligned}
    \]
    holds for any $k_1>0,k_2>0$, which completes the second part of the proof.
\end{proof}
\begin{remark}\label{remark_mu_2}
Compared to \cite[Lemma 2]{parsegov2012nonlinear}, where
\[
\begin{aligned}
&\alpha_1 = 1-\frac{1}{2\mu}, \quad \alpha_2 = 1+\frac{1}{2\mu}, \quad \mu>1 \\
\Leftrightarrow~& \alpha_1 = 1-\frac{1}{\mu}, \quad \alpha_2 = 1+\frac{1}{\mu}, \quad \mu>2,
\end{aligned}
\]
Eqn. \eqref{eqn_alpha_1_2} of Proposition 
\ref{lemma_fixed_time_stable_simplified} permits the choice $\mu=2$. In particular, $\mu=2$ corresponds to the Euclidean norm of a vector, which significantly simplifies the physical implementation in analog hardware.
\end{remark}

\begin{proposition}\label{proposition_unconstrained}
Design a Newton-based fixed-time-stable ODE by
    \begin{equation}\label{eqn_FT_Newton}
    \left(\nabla^2f(x)\right) \dot{x}=-k\nabla f(x) \cdot \left(\frac{1}{\|\nabla f(x)\|} + \|\nabla f(x)\|\right)
\end{equation}
for corresponding to the unconstrained optimization \eqref{eqn_unconstrained_problem}. Given a prescribed settling time $T_p$ and choosing $k=\frac{\pi}{2T_p}$, the solution trajectory of ODE \eqref{eqn_FT_Newton} is fixed-time-stable to an optimal solution $x^*$ of the unconstrained optimization \eqref{eqn_unconstrained_problem} within the settling time $T_p$ for all $x(0)\in\rr^{n_x}$.
\end{proposition}

\begin{proof}
    By introducing $y=\nabla f(x)$, we have that $\dot{y} = \nabla^2f(x) \dot{x}$ and ODE \eqref{eqn_FT_Newton} is equivalent to
\begin{equation}\label{eqn_FT_Newton_equivalent}
    \dot{y}=-ky\cdot \left(\frac{1}{\|y\|}+\|y\| \right).
\end{equation}
Now consider the Lyapunov function $V(y)=\tfrac{1}{2}\|y\|^2$, which is radially unbounded. The time derivative $\dot{V}(y)$ along the trajectories of ODE \eqref{eqn_FT_Newton_equivalent} is
\[
\begin{aligned}
\dot{V}(y)&=y^\top \dot{y}=-k\|y\|-k\|y\|^{3}\\
&= -k (2V)^{1/2}-k (2V)^{3/2}<0,    
\end{aligned}
\]
which is the special case: $k_1=\sqrt{2}k,~k_2=2\sqrt{2}k,~\alpha_1=\frac{1}{2},~\alpha_2=\frac{3}{2}$ in Lemma \ref{lemma_fixed_time_stable} and $k_1=\sqrt{2}k,~k_2=2\sqrt{2}k,~\mu=2$ in Lemma \ref{lemma_fixed_time_stable_simplified}. Thus, after choosing $k=\frac{\pi}{2T_p}$, the settling time will always be smaller than the prescribed (desired) time by Lemma \ref{lemma_fixed_time_stable_simplified}, which completes the proof.
\end{proof}

\section{Method}
\subsection{Transforming non-smooth Lasso into smooth nonnegative QP}
The Lasso problem augments a least-squares objective with an $\ell_1$-regularization term, 
\begin{equation}\label{eqn_Lasso}
    \min_{x\in\mathbb{R}^{n_x}} \|Ax-b\|_2^2 + \tau \|x\|_1,
\end{equation}
where $A\in\mathbb{R}^{m\times n_x}, b\in\mathbb{R}^{m}$ is the problem data, and $\tau>0$ is a fixed regularization parameter used to trade off between quality of fit (small $\tau$) and sparsity of the solution of $x$ (large $\tau$). 

\begin{remark}
(Uniqueness of the minimizer): 
Assume that $A$ has full column rank. Then the Lasso problem \eqref{eqn_Lasso} admits a unique minimizer, since the positive definiteness of $A^\top A$ implies that the quadratic term $\|Ax-b\|_2^2$ is strictly convex, and hence the overall objective function is strictly convex.

If $A$ is not a full rank matrix, an alternative way to guarantee the uniqueness of the minimizer is to introduce an additional $\ell_2$-regularization term,
\begin{equation}\label{eqn_elastic_net}
    \min_{x\in\mathbb{R}^{n_x}} \|Ax-b\|_2^2 + \tau \|x\|_1 + \rho \|x\|_2^2,
\end{equation}
where $\rho>0$. The resulting $\ell_2$-regularized Lasso problem \eqref{eqn_elastic_net} is known as the Elastic Net problem \cite{zou2005regularization}, which clearly admits a unique minimizer.
\end{remark}

The non-smooth $\ell_1$ term in Lasso is not differentiable at $x_i=0~\forall i=1,\cdots,n_x$, where the derivative changes abruptly from $-1$ to $1$, making the second derivative undefined at this point. Consequently, Proposition \ref{proposition_unconstrained} is not directly applicable to the Lasso problem. To handle the non-smooth $\ell_1$ term, we first demonstrate that $\ell_2$-regularized Lasso \eqref{eqn_elastic_net} can be equivalently transformed into a non-negative QP as stated in the following proposition.
\begin{proposition}\label{lemma_Lasso_QP}
    Consider the $\ell_2$-regularized Lasso \eqref{eqn_elastic_net}, let $g:\rr^{n_x}\times \rr^{n_x}\rightarrow\rr$ be defined by
    \begin{equation}
    \begin{aligned}
        g(x_{+},x_{-})&\triangleq \|A(x_{+}-x_{-})-b\|_2^2 +\tau [1 \cdots{} 1]\left[\begin{array}{c}
             x_{+}  \\
             x_{-}
        \end{array}\right] \\
        &\quad + \rho \|x_{+}\|_2^2  + \rho \|x_{-}\|_2^2,
    \end{aligned}
    \end{equation}
    and consider the non-negative QP
    \begin{equation}\label{eqn_NNQP}
        \min_{x_{+},x_{-}\geq0}\, g(x_{+},x_{-}).
    \end{equation}
    Then a unique minimizer $(x_{+}^*,x_{-}^*)$ of non-negative QP \eqref{eqn_NNQP} satisfies the complementarity condition: $x_{+,i}^*\cdot x_{-,i}^*=0~\forall i=1,\cdots{},n_x$.
    
    Furthermore, $x^*\triangleq x_{+}^*-x_{-}^*$ is the unique minimizer of Lasso \eqref{eqn_elastic_net}.
\end{proposition}

\begin{proof}
The non-negative QP \eqref{eqn_NNQP} can be written in the standard QP formulation
    \begin{equation}\label{eqn_NNQP_expanded}
        \begin{aligned}
            \min_{x_+\in\mathbb{R}^{n_x},x_-\in\mathbb{R}^{n_x}}&\quad \frac{1}{2}\left[\begin{array}{c}
                 x_+  \\
                 x_- 
            \end{array}\right]^\top Q   \left[\begin{array}{c}
                 x_+  \\
                 x_- 
            \end{array}\right] + \left[\begin{array}{c}
                 x_+  \\
                 x_- 
            \end{array}\right]^\top q\\
            \text{s.t.}&\quad x_+\geq0, ~x_-\geq0
        \end{aligned}
    \end{equation}
    where
    \begin{equation}\label{eqn_Q_q}
        \begin{aligned}
        &Q\triangleq\left[\begin{array}{cc}
                A^\top A &  -A^\top A\\
                -A^\top A & A^\top A
            \end{array}\right] +  \left[\begin{array}{cc}
                \rho I &  \\
                 & \rho I
            \end{array} \right],\\
        &q\triangleq \left[\begin{array}{c}
                 A^\top b   \\
                 -A^\top b  
            \end{array}\right] + \tau \left[\begin{array}{c}
                 \mathbf{1}_{n_x}  \\
                  \mathbf{1}_{n_x} 
            \end{array} \right],
    \end{aligned}    
    \end{equation}
    and $\mathbf{1}_{n_x}:=[1 \cdots 1]^\top\in\rr^{n_x}$. Clearly, $Q\succ0$ and thus the non-negative QP \eqref{eqn_NNQP} admits a unique minimizer.

The proof is by contradiction.
    Assume that there are some indices $i,~1\leq i \leq n_x$, $x_{+,i}^* x_{-,i}^*\neq 0$. Since $x_{+,i}^*,~ x_{-,i}^*\geq0$, this implies that $x_{+,i}^* x_{-,i}^*>0$. Then, we can construct a solution $(\Bar{x}_{+},\Bar{x}_{-})$ defined by
    \[
    \Bar{x}_{+}\triangleq x_{+}^* -\alpha e_i,\quad \Bar{x}_{-}\triangleq x_{-}^* -\alpha e_i,
    \]
    where $\alpha\triangleq \min\{x_{+,i}^*,~ x_{-,i}^*\}>0$ and $e_i$ is the $i$th column of the identity matrix. Thus, we have that
    \[
    \Bar{x}_+-\Bar{x}_-=x_+^*-x_-^*\Rightarrow \|A(\Bar{x}_+-\Bar{x}_-)-b\|_2^2=\|A(x_+^*-x_-^*)-b\|_2^2
    \]
    and
    \[
    \begin{aligned}
    \tau [1\cdots{}1] \left[\begin{array}{c}
         \Bar{x}_+  \\
         \Bar{x}_- 
    \end{array}\right]&=\tau[1\cdots{}1] \left[\begin{array}{c}
         x_+^*  \\
         x_-^* 
    \end{array}\right] -2\tau n_x\alpha\\  
    &<\tau[1\cdots{}1] \left[\begin{array}{c}
         x_+^*  \\
         x_-^* 
    \end{array}\right].
    \end{aligned}
    \]
    If we can prove that $\|\Bar{x}_+\|_2^2\leq \|x_+^*\|_2^2 $ and $\|\Bar{x}_-\|_2^2\leq \|x_-^*\|_2^2$, then we have that
    \[
       g(\Bar{x}_+,\Bar{x}_{-})< g(x_+^*,x_-^*),
   \]
    which contradicts that $(x_{+}^*,x_{-}^*)$ is the unique minimizer of non-negative QP \eqref{eqn_NNQP}, and therefore the initial assumption $x_{+,i}^*x_{-,i}^*\neq 0$ does not hold.
    
    To prove the two inequalities, choosing 
    \[
    \beta_{x_+}\triangleq\frac{\alpha}{x_{+,i}^*},~\beta_{x_-}\triangleq\frac{\alpha}{x_{-,i}^*} \Rightarrow \beta_{x_+}\in(0,1],~\beta_{x_-}\in(0,1],
    \]
    implies that
    \[
    \begin{aligned}
    \Bar{x}_{+}&=(1-\beta_{x_+}) x_{+}^* + \beta_{x_+}(x_{+}^* -\alpha e_i),\\
    \Bar{x}_{-}&=(1-\beta_{x_-}) x_{-}^* + \beta_{x_+}(x_{-}^* -\alpha e_i).
    \end{aligned}
    \]
    Since the regularization term $\|\cdot\|_2^2$ is convex and satisfies Jensen's inequality,
    \[
    \begin{aligned}
    \|\Bar{x}_+\|_2^2 &= \|(1-\beta_{x_+}) x_{+}^* + \beta_{x_+}(x_{+}^* -\alpha e_i) \|_2^2\\
    &\leq (1-\beta_{x_+})\|x_{+}^*\|_2^2 +  \beta_{x_+}\|x_{+}^* -\alpha e_i\|_2^2     \\
    &=\|x_+^*\|_2^2 - \beta_{x_+}\left(\|x_+^*\|_2^2 - \|x_{+}^* -\alpha e_i\|_2^2\right)\\
    &\leq\|x_+^*\|_2^2.
    \end{aligned}
    \]
   A similar analysis gives that $\|\Bar{x}_-\|_2^2 \leq\|x_-^*\|_2^2$. Thus, the first part of the proof is completed.

 The second part of the proof is also by contradiction. Let $x^*=x_{+}^*-x_{-}^*$ and assume that $x^*\neq\argmin$ \eqref{eqn_elastic_net} (not the minimizer of Lasso \eqref{eqn_elastic_net}). Let $\Bar{x}=\argmin$ \eqref{eqn_elastic_net} (the unique minimizer of Lasso \eqref{eqn_elastic_net}). Setting
    \[
    \Bar{x}_+\triangleq \max(\Bar{x},0),\quad \Bar{x}_-\triangleq \max(-\Bar{x},0)
    \]
    gives that $\Bar{x}=\Bar{x}_+-\Bar{x}_-$.
    Since $\Bar{x}_{+,i}\Bar{x}_{-,i}=0,~\forall i=1,\cdots{},n_x$, we have that
    \[
    \begin{aligned}
        g(\Bar{x}_+,\Bar{x}_-) &=  \|A\Bar{x}-b\|_2^2+\tau\|\Bar{x}\|_1 + \rho \|\Bar{x}\|_2^2 \\
        &<\|Ax^*-b\|_2^2+\tau\|x^*\|_1 + \rho \|x^*\|_2^2
    \end{aligned}
    \]
    which contradicts $x^*\neq\argmin$ \eqref{eqn_elastic_net} and thus completes the second part of the proof.
\end{proof}

In addition to yielding a smooth QP, introducing nonnegative constraints is also necessary because analog hardware typically supports only nonnegative variables. For instance, voltages are inherently nonnegative.

\subsection{Transforming non-negative QP into fixed-time-stable Newton ODE}
The Karush–Kuhn–Tucker (KKT) condition \cite[Ch.\ 5]{boyd2004convex} for the non-negative QP \eqref{eqn_NNQP_expanded} is
\begin{subequations}\label{eqn_KKT}
    \begin{align}
        & Q\left[\begin{array}{c}
                 x_+  \\
                 x_- 
            \end{array}\right]-\left[\begin{array}{c}
                 v \\
                 s 
            \end{array}\right] + q =0,\label{eqn_KKT_a}\\
        &\left[\begin{array}{c}
                 x_+  \\
                 x_- 
            \end{array}\right] \pdot \left[\begin{array}{c}
                 v \\
                 s
            \end{array}\right]=0,\label{eqn_KKT_complementarity}\\
            &\left[\begin{array}{c}
                 x_+  \\
                 x_- 
            \end{array}\right]\geq0, \left[\begin{array}{c}
                 v  \\
                 s
            \end{array}\right]\geq0,\label{eqn_KKT_non_negative}
    \end{align}
\end{subequations}
where $v\in\rr^{n_x}, s\in\rr^{n_x}$ denote the Lagrangian dual variable for inequality constraints $x_+\geq0$ and $x_-\geq0$, respectively. The symbol $\pdot$ denotes the Hadamard product of two vectors.
For simplicity, adopt the notations 
\begin{equation}
    z\triangleq\left(\begin{array}{c}
         x_+  \\
         x_- 
    \end{array}\right)\in\rr^{2n_x}, \quad w\triangleq\left(\begin{array}{c}
         v\\
         s
    \end{array}\right)\in\rr^{2n_x}.
\end{equation}
KKT \eqref{eqn_KKT} constitutes a set of nonlinear equations with non-negativity constraints, and the Newton-like interior-point methods (IPMs) generally do not achieve the same computational efficiency as Newton methods in unconstrained nonlinear equations, because the step size must be carefully controlled to ensure that iterates remain within the non-negative bounds.

\begin{remark}
    \textbf{Continuous interior-point flow}: here we argue that \textbf{the non-negativity constraints \eqref{eqn_KKT_non_negative} can be naturally satisfied in continuous time}, rather than in a discrete iterative setting, due to the existence of the complementarity constraint \eqref{eqn_KKT_complementarity}, as formally proved in Proposition \ref{lemma_non_negative}. Moreover, a continuous-time interior-point flow derived from the central-path theory of IPMs \cite{wright1997primal} is: for any $\theta>0$ starting from $(z(0),w(0))$, there is a unique strictly positive point $(z(\theta),w(\theta))$ for KKT \eqref{eqn_KKT}, such that
\[
    \begin{aligned}
        & Qz(\theta)-w(\theta)+ q =\theta\left(Qz(0)-w(0) + q\right),\\
        &z(\theta) \pdot w(\theta)=\theta(z(0) \pdot w(0)),
    \end{aligned}
\]
and the trajectory $(z(\theta),w(\theta))$ is a continuous bounded trajectory and, when $\theta\rightarrow0$, the limit point $(z(\theta),w(\theta))$ is the solution of KKT \eqref{eqn_KKT}.
\end{remark}

Also inspired by Proposition \ref{proposition_unconstrained}, we convert KKT \eqref{eqn_KKT}, specifically \eqref{eqn_KKT_a} and \eqref{eqn_KKT_complementarity}, without explicitly enforcing \eqref{eqn_KKT_non_negative}, into a Newton-based fixed-time-stable ODE, 
\begin{equation}\label{eqn_KKT_fixed-time-stable_ODE}
\begin{aligned}
Q \dot{z} - \dot{w}&=-k\frac{Qz-w+q }{\left\|\left[\begin{array}{c}
Qz-w+q \\
z\pdot w
\end{array}\right]\right\|} \\
&\quad- k\frac{Qz-w+q}{\left\|\left[\begin{array}{c}
Qz-w+q  \\
z\pdot w
\end{array}\right]\right\|^{-1}},\\
W \dot{z} + Z\dot{w} &=-k\frac{z\pdot w}{\left\|\left[\begin{array}{c}
Qz-w+q  \\
z\pdot w
\end{array}\right]\right\|}\\
&\quad- k\frac{z\pdot w}{\left\|\left[\begin{array}{c}
Qz-w+q  \\
z\pdot w
\end{array}\right]\right\|^{-1}},
\end{aligned}
\end{equation}
where $k>0$ and $Z\triangleq \diag(z),W\triangleq\diag(w)$.

\begin{proposition}\label{lemma_uniqueness}
 Given an initial point $(z(0),w(0))$, the trajectory $(z(t),w(t))$ of ODE \eqref{eqn_KKT_fixed-time-stable_ODE} exists and is unique for all $t\geq0$. Furthermore, for
\begin{equation}
    k=\frac{\pi}{2T_p},
\end{equation}
the settling time of ODE \eqref{eqn_KKT_fixed-time-stable_ODE} is bounded by $T_p$.
\end{proposition}
\begin{proof}
    First, by introducing 
    \[
    u_1 \triangleq Qz-w+q,\quad u_2\triangleq z\pdot w,
    \] 
we have that $\dot{u}_1=Q \dot{z} - \dot{w},~\dot{u}_2= W \dot{z} + Z\dot{w}$. Then, by letting $u\triangleq\mathrm{col}(u_1,u_2)$, ODE \eqref{eqn_KKT_fixed-time-stable_ODE} is equivalently reformulated as
\begin{equation}\label{eqn_ODE_u_simplified}
\dot{u}=-k\frac{u}{\|u\|}-k\frac{u}{\|u\|^{-1}}.
\end{equation}
Thus, we turn to proving that ODE \eqref{eqn_ODE_u_simplified} has a unique solution and is \textit{fixed-time-stable}. Now, consider the Lyapunov function $V(u)=\frac{1}{2}\|u\|^2$, which is radially unbounded. Proving the existence and uniqueness of a solution for ODE \eqref{eqn_ODE_u_simplified} by Okamura's Uniqueness Theorem (see \cite[Thm.\ 3.15.1]{agarwal1993uniqueness}) is equivalent to proving that: \textit{i)} $V(0)=0$, \textit{ii)} $V(u)>0$ if $\|u\|\neq0$, \textit{iii)} $V(u)$ is locally Lipschitz, and \textit{iv)} $\dot{V}(u)\leq0$. By the definition of $V(u)$, it is clear that \textit{i)}, \textit{ii)}, and \textit{iii)} hold. To prove \textit{iv)}, the time derivative $\dot{V}(w)$ along the trajectories of ODE \eqref{eqn_ODE_u_simplified} is
\begin{equation}
    \begin{aligned}
         \dot{V}&=u^\top \dot{u}=-k\|u\|^{2-1}-k\|u\|^{2+1}\\
         &= -k (2V)^{1/2} -k (2V)^{3/2}<0,   
    \end{aligned}
\end{equation}
which completes the first part of the proof. Then, by Proposition \ref{lemma_fixed_time_stable_simplified} and Remark \ref{remark_mu_2} where $\mu=2$, the upper bound for the settling time is 
    \[
T_{\max}=\frac{\pi}{\sqrt{k(2)^{\frac{1}{2}}k(2)^{\frac{3}{2}}}} = \frac{\pi}{2k}=T_p,
    \]
    which completes the second part of the proof.
\end{proof}

\begin{remark}
The trajectory of ODE \eqref{eqn_ODE_u_simplified} cannot recover $(z,w)$; the trajectory of ODE \eqref{eqn_KKT_fixed-time-stable_ODE} is still required, see Alg.\ \ref{alg1}.
\end{remark}

\begin{proposition}\label{lemma_non_negative}
    Given an initial point $(z(0),w(0))>0$, the trajectory $(z(t),w(t))$ of ODE \eqref{eqn_KKT_fixed-time-stable_ODE} always satisfies $z(t)\geq0$, $w(t)\geq0$. 
\end{proposition}

\begin{proof}
The proof is by contradiction. Given an initial point $(z(0),w(0))>0$, we have that $u_2(0)=z(0)\pdot w(0)>0$. According to the Forward Invariance Theorem (Nagumo’s Theorem) of barrier functions (see   \cite{glotfelter2017nonsmooth}), $u_2(t)\geq0$ always holds if $u_2(0)>0$, which implies that $z(t)\geq0$, $w(t)\geq0$ or $z(t)\leq0$, $w(t)\leq0$. Let $\tau$ be defined as $\tau = \inf\{t: z(t) <0, w(t) < 0\}$, so that for all $t\leq \tau$, $z(t) \geq 0$ and $w(t) \geq 0$. Now, for such a time instant $\tau$ to exist, it is necessary that the trajectory of $z$ or $w$ reach the origin and leave the origin so that either $z$ or $w$ can switch signs. As a result, once the trajectories of $u_2$ reach the origin, they remain at the origin. Thus, $z(t)$ and $w(t)$ cannot change the sign and, as a result, there exists no such $\tau$, which completes the proof.
\end{proof}

Combining Proposition \ref{lemma_Lasso_QP}, \ref{lemma_uniqueness}, and  \ref{lemma_non_negative} leads to the next theorem.
\begin{theorem}\label{thm_1}
    Given an initial point $(z(0),w(0))>0$, a prescribed settling time $T_p$, and choosing 
\begin{equation}
k=\frac{\pi}{2T_p}
\end{equation}
    for ODE \eqref{eqn_KKT_fixed-time-stable_ODE}, then its trajectory $(z(t),w(t))$ arrives the equilibrium point at time $T_p$ in the worst case and 
    \[
    x^*=z_{1:n_x}(T_p)-z_{n_x+1:2n_x}(T_p)
    \]
    is a unique optimal solution of Lasso \eqref{eqn_elastic_net}.
\end{theorem}

\subsection{Discussion about Implementation}
Summarizing the above procedures, the proposed method is presented in Algorithm \ref{alg1}. Step 3 can be implemented using either analog or digital computers. As claimed in \cite{wu2025Arbitrarily}, solving the fixed-time-stable ODE \eqref{eqn_KKT_fixed-time-stable_ODE} using analog computers offers two key advantages: \textbf{i) ultra-low energy consumption}, and \textbf{ii) arbitrarily small execution time independent of the problem dimension}.

\begin{algorithm}[H]
    \caption{Solving Lasso \eqref{eqn_elastic_net} via fixed-time-stable ODE}\label{alg1}
    \textbf{Input}: Given $A,b,\tau,\rho$ and prescribed settling time $T_p$
    \vspace*{.1cm}\hrule\vspace*{.1cm}
    \begin{algorithmic}[1]
    \STATE Define $Q,q$ by \eqref{eqn_Q_q}, choose $k = \frac{\pi}{2T_p}$, and construct ODE \eqref{eqn_KKT_fixed-time-stable_ODE};
    \STATE Initialize $(z(0),w(0))>0$ such as $z(0)=\mathbf{1}_{2n_x},w(0)=\mathbf{1}_{2n_x}$;
    \STATE Map ODE \eqref{eqn_KKT_fixed-time-stable_ODE} into the physical system (analog computer) or simulate ODE \eqref{eqn_KKT_fixed-time-stable_ODE} via numerical integrators in digital computers;
    \STATE Collect the values of the ODE state at time $T_p$: $z(T_p)$
    \STATE Recover the solution of Lasso \eqref{eqn_elastic_net}: $x^*=z_{1:n_x}(T_p)-z_{n_x+1:2n_x}(T_p)$.
    \end{algorithmic}
\end{algorithm}

Despite the recent progress in analog hardware, including advances in reconfigurability and programmability \cite{achour2016configuration}, it remains relatively less mature than digital computing platforms. Consequently, this article validates the algorithm through simulations on digital computers, while experimental verification using analog hardware is currently underway. 

\begin{figure*}[!htbp]
\begin{picture}(140,200)
\put(0,0){\includegraphics[width=90mm]{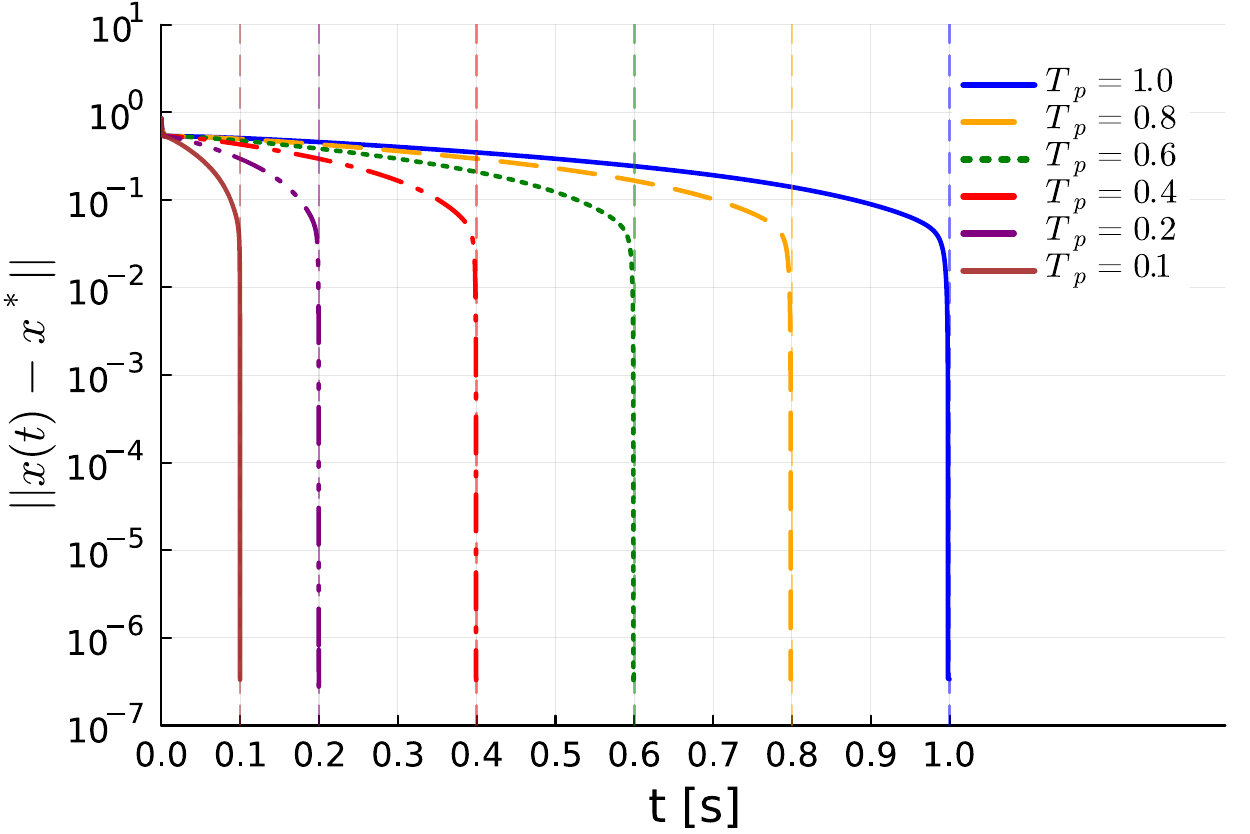}}
\put(260,0){\includegraphics[width=90mm]{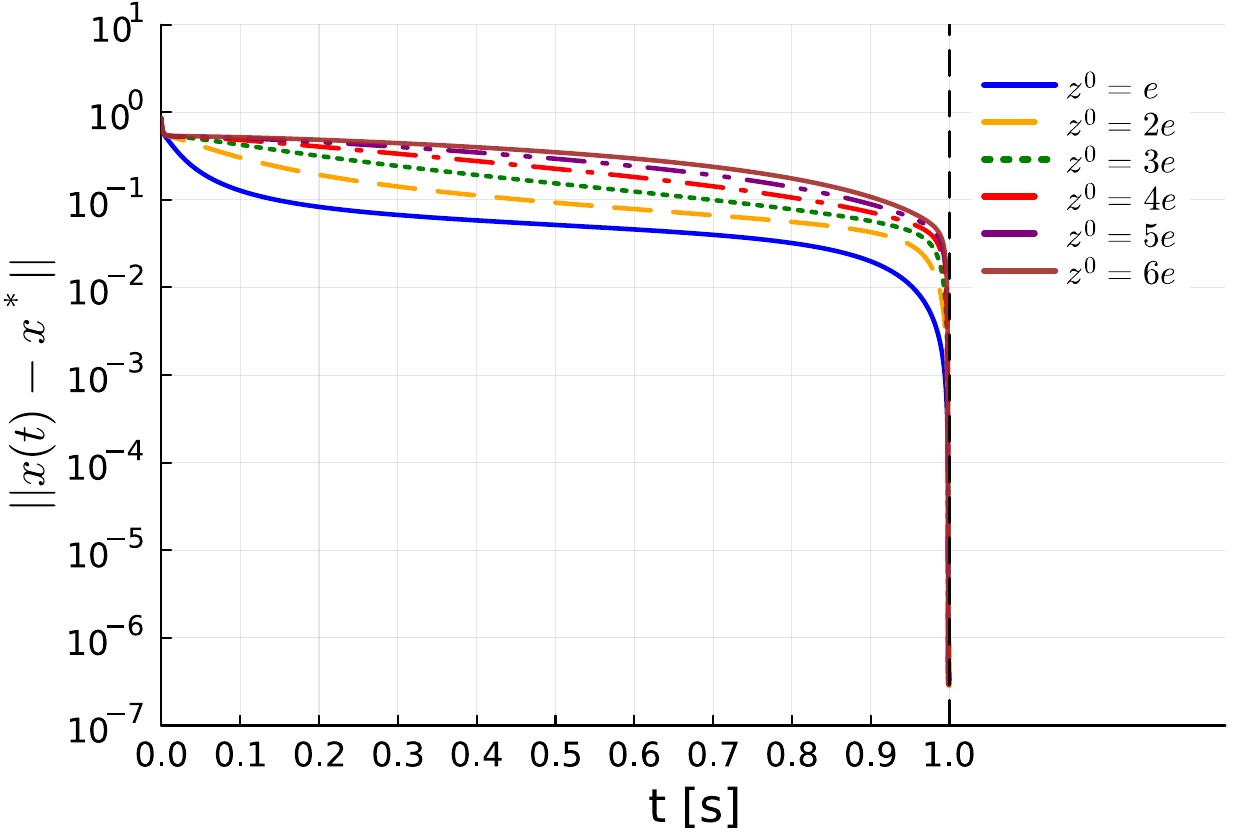}}
\end{picture}
\caption{Left: the trajectory $\|x(t)-x^*\|$ of Experiment (1) for one Lasso example; Right: the trajectory $\|x(t)-x^*\|$ of Experiment (2) for one Lasso example.}
\label{fig_result}
\end{figure*}

\section{Numerical Examples}
The main contribution of this article is Thm.\ \ref{thm_1}, which establishes how the settling time of the transformed ODE associated with the Lasso problem can be explicitly prescribed. This article is the first part of a series of studies, focusing on building the theoretical foundation and providing a ``soft-verification'' of the proposed approach. Specifically, ODE \eqref{eqn_KKT_fixed-time-stable_ODE} is numerically simulated to validate the correctness of Thm.\ \ref{thm_1}. The simulations are conducted using Julia’s \texttt{DFBDF()} ODE solver on a MacBook Pro with an Apple M4 Max processor and 36 GB RAM. The second part (the “hard verification”), which focuses on the analog hardware implementation and a tailored printed circuit board (PCB) for the Lasso problem, is currently in progress.

We generate 100 Lasso problems with random $A,b$ with $n_x=10$, $m=20$, $\tau=1.0$, and $\rho=0.1$. Two experiments are conducted. \textit{Experiment (1)} tests prescribed settling times $T_p=1,0.8,0.6,0.4,0.2,0.1$ by setting $k=\pi/(2T_p)$ with the same initial condition $(z(0),w(0))=e$ (the vector of ones). \textit{Experiment (2)} tests different initial conditions $(z(0),w(0))=ie$, $i=1,\dots,6$, with $T_p=1$. Figure \ref{fig_result} shows the results, confirming that the settling time can be arbitrarily prescribed and is independent of the initial condition, consistent with Theorem ~\ref{thm_1}.

\section{Conclusion}
This article developed a fixed-time-stable ODE representation for the Lasso problem. The approach relies on transforming the non-smooth Lasso formulation into a smooth nonnegative QP and constructing a Newton-based fixed-time-stable ODE for solving its KKT conditions. The resulting ODE automatically enforces nonnegativity constraints and guarantees convergence to the optimal solution within a prescribed time bound that is independent of the initial condition and problem data.  Numerical simulations confirm the theoretical settling-time property. 

Future work will focus on hardware verification using large-scale custom printed circuit boards, and the resulting Lasso circuits offer energy-efficient and scalable solutions. 
Since analog hardware suffers from device mismatch, finite precision, and noise, future work will also study fixed-time stability under such perturbations and develop Lyapunov-based noise-aware designs.

\bibliographystyle{IEEEtran}
\bibliography{Ref}
\end{document}